\RequirePackage[l2tabu, orthodox]{nag} %before \documentclass
\documentclass[a4paper,reqno,11pt]{article} %default font size: 10pt %equation numbers on the right
\usepackage[stretch=10]{microtype} %micro layout engine
\usepackage[german,french,UKenglish]{babel}
\usepackage[T1]{fontenc} %needed for French
\usepackage[utf8]{inputenc} %before csquotes
\usepackage{csquotes} %quotations in different languages
\usepackage{letltxmacro} %to fix the parskip before proofs
\usepackage{amsmath,amsfonts,amssymb,amsthm,mathrsfs}
\usepackage[margin=1.3in]{geometry}
\usepackage{setspace}
\usepackage{enumitem}
\usepackage{calc}
\usepackage{makebox}
\usepackage{xcolor}
\usepackage{xspace}
\usepackage{xifthen}
\usepackage{titlesec} %for section title spacing - breaks nameref in hyperref
\usepackage{tikz}
\usepackage{tikz-cd} %commutative diagrams in TikZ
\usetikzlibrary{calc}
\usetikzlibrary{babel}
\usetikzlibrary{matrix,arrows}
\usetikzlibrary{decorations.markings}
\usepackage[lf]{Baskervaldx} % lining figures
\usepackage[bigdelims,vvarbb]{newtxmath} % math italic letters from Nimbus Roman
\usepackage[cal=boondoxo]{mathalfa} % mathcal from STIX, unslanted a bit

\usepackage{verbatim} %includes multiline comment
\usepackage{aliascnt} %for creating different hyperref references for different theoremstyles
\usepackage[unicode,colorlinks=true,urlcolor=blue!70!black,citecolor=blue!60!black,linkcolor=blue!60!black,pdfauthor={Remy van Dobben de Bruyn},pdftitle={Divisors in projective bundles over the projective line whose complement is affine space}]{hyperref} %always add last

\titlespacing{\section}{0pt}{12pt plus 6pt minus 4pt}{0pt plus 4pt minus 2pt}
\titlespacing{\subsection}{0pt}{12pt plus 6pt minus 4pt}{0pt plus 4pt minus 2pt}
\titlespacing{\subsubsection}{0pt}{0pt plus 4pt minus 2pt}{0pt plus 3pt minus 2pt}

\titleformat{\section}[block]{\Large\bfseries\scshape\filcenter}{\thesection.}{1ex}{}
\titleformat{\subsection}{\large\scshape\filcenter}{\thesubsection}{1ex}{}

\numberwithin{equation}{section}
\newtheoremstyle{thms}{1em}{0pt}{\itshape}{}{\itshape\bfseries}{. ----}{ }{\thmname{#1}\xspace\thmnumber{#2}\thmnote{ \normalfont(#3)}}
\theoremstyle{thms}
\newaliascnt{Thm}{equation}							%Theorem
\newtheorem{Thm}[Thm]{Theorem}
\aliascntresetthe{Thm}

\newaliascnt{Prop}{equation}						%Proposition

\aliascntresetthe{Prop}

\newaliascnt{Lemma}{equation}						%Lemma
\newtheorem{Lemma}[Lemma]{Lemma}
\aliascntresetthe{Lemma}

\newaliascnt{Cor}{equation}							%Corollary
\newtheorem{Cor}[Cor]{Corollary}
\aliascntresetthe{Cor}

\newaliascnt{Conj}{equation}						%Conjecture

\aliascntresetthe{Conj}

\newaliascnt{Question}{equation}					%Question

\aliascntresetthe{Question}

\newtheoremstyle{defs}{1em}{0pt}{}{}{\itshape\bfseries}{. ----}{ }{\thmname{#1} \thmnumber{#2}}
\theoremstyle{defs}
\newaliascnt{Rmk}{equation}							%Remark

\aliascntresetthe{Rmk}

\newaliascnt{Fact}{equation}						%Fact

\aliascntresetthe{Fact}

\newaliascnt{Def}{equation}							%Definition
\newtheorem{Def}[Def]{Definition}
\aliascntresetthe{Def}

\newaliascnt{Ex}{equation}							%Example

\aliascntresetthe{Ex}

\newaliascnt{Con}{equation}							%Construction

\aliascntresetthe{Con}

\newaliascnt{Not}{equation}							%Notation

\aliascntresetthe{Not}

\newaliascnt{Setup}{equation}						%Setup

\aliascntresetthe{Setup}

\newaliascnt{Picture}{equation}						%Picture

\aliascntresetthe{Picture}

\newtheoremstyle{par}{1em}{0pt}{}{}{\itshape\bfseries}{. ----}{ }{\thmnumber{#2}}
\theoremstyle{par}
\newaliascnt{Par}{equation}							%Paragraph

\aliascntresetthe{Par}

\theoremstyle{thms}
\newtheorem{thm}{Theorem}
\newtheorem*{thm*}{Theorem}

\newtheorem*{lemma*}{Lemma}

\LetLtxMacro\oldproof\proof	%Proof without extra skip before and after
\renewcommand{\proof}[1][Proof]{\oldproof[#1]\unskip}
\LetLtxMacro\oldendproof\endproof
\renewcommand{\endproof}{\oldendproof\unskip}

\newenvironment{itemize*} %no longer needed because of \setlist below?
  {\begin{itemize}
    \setlength{\itemsep}{1em}
    \setlength{\parskip}{-1em}
    \setlength{\topsep}{0pt}
    \setlength{\partopsep}{0pt}}
  {\end{itemize}}
  
\newenvironment{enumerate*}
  {\begin{enumerate}
    \setlength{\itemsep}{1em}
    \setlength{\parskip}{-1em}
    \setlength{\topsep}{0pt}
    \setlength{\partopsep}{0pt}}
  {\end{enumerate}}
  
\setlist{itemsep=0em,topsep=0cm,partopsep=0em,parsep=\lineskip}

\setlist[enumerate]{label=\normalfont(\alph*), align=left, leftmargin=3em, labelwidth=1em, itemindent=0pt, listparindent=0pt, labelindent=1em, labelsep=*}

\setlist[itemize]{leftmargin=1.3em}

\DeclareMathOperator{\Spec}{Spec}
\DeclareMathOperator{\Sym}{Sym}
\DeclareMathOperator{\Pic}{Pic}
\DeclareMathOperator{\Hom}{Hom}
\DeclareMathOperator{\Bl}{Bl}

\newcommand{\tors}{_{\operatorname{tors}}}
\newcommand{\tf}{_{\operatorname{tf}}}

\newcommand{\punct}[1]{\makebox[0pt][l]{\,#1}} %so that full stops do not mess with tikz layout

\tikzcdset{arrow style=math font}

\let\OLDthebibliography\thebibliography
\renewcommand\thebibliography[1]{
  \OLDthebibliography{#1}
  \setlength{\parskip}{0pt}
  \setlength{\itemsep}{0pt plus 0.1em}
}

\begin{document}

\renewcommand{\sectionautorefname}{Section}
\renewcommand{\subsectionautorefname}{Subsection}		%Subsection

\begin{center}
\vspace*{-3.7em}
\noindent\makebox[\linewidth]{\rule{15cm}{0.4pt}}
\vspace{-.6em}

{\LARGE{\textbf{Divisors in projective bundles over the projective line whose complement is affine space}}}

\vspace{.4em}

{\textsc{Remy van Dobben de Bruyn}}

\vspace*{.2em}
\leavevmode
\noindent\makebox[\linewidth]{\rule{11cm}{0.4pt}}

\vspace{.5em}
\end{center}

\renewcommand{\abstractname}{\small\bfseries\scshape Abstract}

\begin{abstract}\noindent% should not contain macros when uploaded to arXiv
Given a smooth projective variety~$X$ of dimension~$n$ and a closed subscheme \makebox{$Z \subseteq X$}, it is in general a difficult problem to determine whether $X \setminus Z$ is isomorphic to~$\mathbf A^n$. In the case where~$X$ is a projective bundle over~$\mathbf P^1$ and the restriction of every irreducible component of~$Z$ to every fibre is a hyperplane, we give a complete geometric characterisation in terms of the irreducible components of~$Z$ and their intersections. In the appendix, we use this to obtain a coordinate-free explanation for the recent first counterexample to the Jacobian conjecture.
\end{abstract}\vspace{.2em}

\section{Introduction}
The recognition of affine spaces~$\mathbf A^n$ is an old and well-studied subject; see for instance \cite{Ramanujam}, \cite{MiyanishiAffine2}, \cite{MiyanishiAffine3}, \cite{DasguptaGupta} for general results in low dimension. In this paper, we explain a version for complements of divisors in projective bundles~$\mathbf P(\mathscr E)$ over~$\mathbf P^1$ (over an arbitrary field). Recalling that $\Pic(\mathbf P(\mathscr E)) \cong \mathbf Zf \oplus \mathbf Zh$ where $f = \pi^*\mathcal O_{\mathbf P^1}(1)$ is a fibre and $h = \mathcal O_{\mathbf P(\mathscr E)}(1)$ is the relative tautological bundle, we prove the following result in the case of divisors whose restriction to each fibre is a hyperplane:

\begin{thm}\label{thm main}
Let~$\mathscr E$ be a vector bundle of rank~$n \geq 2$ on~$\mathbf P^1$, and let $\pi \colon X = \mathbf P(\mathscr E) \to \mathbf P^1$ be the projection. Let $D_1,\ldots,D_s$ be pairwise distinct irreducible divisors on~$\mathbf P(\mathscr E)$ with $D_i \sim d_if + h$ for $d_i \in \mathbf Z$, and let $U = X \setminus (D_1 \cup \ldots \cup D_s)$. Then $U \cong \mathbf A^n$ if and only if $s = 2$, $\lvert d_2 - d_1 \rvert = 1$, and there is exactly one point $p \in \mathbf P^1$ such that the fibres $D_1 \cap \pi^{-1}(p)$ and $D_2 \cap \pi^{-1}(p)$ agree.
\end{thm}

See \autoref{Thm criterion} and the remarks preceding it.

\subsection*{Relation to the Jacobian conjecture}
On 19 July 2026, Levent Alp\"oge posted a counterexample to the three-dimensional Jacobian conjecture on the social media platform~X, found using AI \cite{AlpogeX}. As explained by a comment of Andy Jiang one day later (also found using AI) \cite{JiangX}, the example can be understood geometrically by taking a ramified cover $\varphi \colon \mathbf P^1 \times \mathbf P^2 \to \mathbf P^3$ and removing a suitably chosen hyperplane~$H$ (resp.~the ramification divisor~$R$ and~$\varphi^{-1}(H)$) from the source (resp.~target). It is clear that~$\mathbf P^3 \setminus H$ is isomorphic to~$\mathbf A^3$, and there was a lot of online activity to understand why $(\mathbf P^1 \times \mathbf P^2) \setminus (R \cup \varphi^{-1}(H))$ is isomorphic to~$\mathbf A^3$; see for instance \cite{LittX} and the blog posts \cite{SBS}, \cite{TaoBlog}, \cite{LevinsonBlog}, as well as the blog posts \cite{vDdBBlog1,vDdBBlog2} and expository paper \cite{vDdBExpo} by the author.

When $n = 3$ and~$\mathscr E$ is the trivial bundle~$\mathcal O^3$, the theorem above gives a criterion on~$R$ and~$\varphi^{-1}(H)$ in order for their complement to be isomorphic to~$\mathbf A^3$. In the appendix, we use \autoref{thm main} to show that these criteria are satisfied if and only if~$H$ is tangent but not osculating to the small diagonal (the same condition appearing in \cite{JiangX})\footnote{The `only if' part is not difficult, but many of the online explanations do not address it, and sometimes only show the existence of a \emph{single} hyperplane~$H$ with the required properties.}. In particular, this gives a coordinate-free explanation of the counterexample of \cite{AlpogeX} and \cite{JiangX}. We include this in the paper because it does not yet seem to appear in the published literature.

\subsection*{Outline of the paper}
In \hyperref[Sec projective]{\S2}, we start with some lemmas on projective bundles for which we do not know a reference, and which may be of independent interest. This is used in \hyperref[Sec criterion]{\S3} to prove \autoref{thm main}. 

Proving necessity of the conditions of the theorem is relatively straightforward, and only uses the cohomology of~$U$ and its class in the Grothendieck ring of varieties. The more surprising fact is that the conditions are in fact sufficient for proving that $U \cong \mathbf A^n$. To prove this, we first reduce to the rank~$2$ case by applying a linear projection. Then~$X$ is a rational surface, and we use ideas from  birational geometry of surfaces to construct an isomorphism $U \cong \mathbf A^2$.

In special cases, the isomorphism $U \cong \mathbf A^n$ was noted in relation to the counterexample to the Jacobian conjecture; see for instance the comments \cite{JiangX} and \cite{LittX} on~X, as well as the blog posts \cite{SBS}, \cite{vDdBBlog1}, \cite{vDdBBlog2}. Thus, \autoref{thm main} puts this construction in a broader context.

\subsection*{Acknowledgements}
This work was supported by ERC Horizon grant no.~101042990. No AI was used by the author in the preparation of this paper (nor in the preparation of its precursors \cite{vDdBBlog1}, \cite{vDdBBlog2}, and \cite{vDdBExpo}). I am grateful to Charles Newton for pointing out a mistake in an earlier version of this paper.

\section{Some lemmas on projective bundles}\label{Sec projective}
If~$S$ is a scheme and~$\mathscr F$ a quasi-coherent sheaf on~$S$, we will write $\mathbf P(\mathscr F) = \mathbf{Proj}_S(\Sym^*\mathscr F)$. Note that this is the projective space of 1-dimensional quotients (rather than subspaces) of~$\mathscr F$: for an $S$-scheme $f \colon T \to S$, morphisms $T \to \mathbf P(\mathscr F)$ over~$S$ correspond to rank~$1$ locally free quotients $f^*\mathscr F \twoheadrightarrow \mathscr L$ on~$T$ \cite[Prop.~4.2.3]{EGA2}. Likewise, we write~$\mathbf A(\mathscr F)$ for $\mathbf{Spec}_S(\Sym^*(\mathscr F))$, whose points on $f \colon T \to S$ are given by maps $f^*\mathscr F \to \mathcal O_T$ (rather than global sections $\mathcal O_T \to f^*\mathscr F$). Below, a \emph{vector bundle} on a scheme~$S$ will mean a locally free sheaf of finite type on~$S$ (this is called a `finite locally free sheaf' in \cite[Tag \href{https://stacks.math.columbia.edu/tag/01C6}{01C6}]{Stacks}).

\begin{Lemma}\label{Lem ses}
Let~$S$ be a scheme, let $0 \to \mathscr E \to \mathscr F \to \mathscr G \to 0$ be a short exact sequence of quasi-coherent sheaves on~$S$, and let $\pi \colon \mathbf P(\mathscr F) \to S$ be the projection. Then the ideal sheaf cutting out the closed immersion $\mathbf P(\mathscr G) \hookrightarrow \mathbf P(\mathscr F)$ is the image of the composite map \makebox{$\pi^*\mathscr E \otimes \mathcal O_{\mathbf P(\mathscr F)}(-1) \to \pi^*\mathscr F \otimes \mathcal O_{\mathbf P(\mathscr F)}(-1) \twoheadrightarrow \mathcal O_{\mathbf P(\mathscr F)}$} induced by the tautological quotient $\pi^*\mathscr F \twoheadrightarrow \mathcal O_{\mathbf P(\mathscr F)}(1)$. If~$\mathscr E$ is invertible, then the map $\pi^*\mathscr E \otimes \mathcal O_{\mathbf P(\mathscr F)}(-1) \to \mathcal O_{\mathbf P(\mathscr F)}$ is injective.
\end{Lemma}

\begin{proof}
The map $\mathbf P(\mathscr G) \hookrightarrow \mathbf P(\mathscr F)$ is induced by the surjection $\Sym^*\mathscr F \to \Sym^*\mathscr G$ of graded algebras, whose kernel is generated by the image of the map $\mathscr E \to \mathscr F \hookrightarrow \Sym^*\mathscr F$ \cite[Ch.~III, \S6, Prop.~4]{BourbakiAlgebre}. Thus, the kernel is the image of the map of quasi-coherent sheaves of graded $\Sym^*\mathscr F$-modules given by $\mathscr E(-1) \otimes \Sym^*\mathscr F \subseteq \mathscr F(-1)\otimes \Sym^*\mathscr F \to \Sym^{\geq 1}\mathscr F$, and this composite map is injective if~$\mathscr E$ is invertible. The result now follows since the functor $\mathscr M \mapsto \widetilde{\mathscr M}$ from quasi-coherent sheaves of graded $\Sym^*\mathscr F$-modules to sheaves on $\mathbf P(\mathscr F)$ is exact and preserves Serre twists \cite[Tag \href{https://stacks.math.columbia.edu/tag/01O7}{01O7} and \href{https://stacks.math.columbia.edu/tag/01MT}{01MT}]{Stacks}.
\end{proof}

\begin{Lemma}\label{Lem torsor}
Let~$S$ be a scheme, let $0 \to \mathscr E \to \mathscr F \to \mathscr G \to 0$ be a short exact sequence of vector bundles on~$S$ such that~$\mathscr E$ has rank~$1$, and let $X = \mathbf P(\mathscr F) \setminus \mathbf P(\mathscr G)$. If $f \colon T \to S$ is a morphism, then~$X(T)$ is the set of splittings of the short exact sequence $0 \to f^*\mathscr E \to f^*\mathscr F \to f^*\mathscr G \to 0$. In particular,~$X$ is naturally a torsor under the $S$-group scheme $\mathbf A(\mathscr E^{-1} \otimes \mathscr G)$.
\end{Lemma}

\begin{proof}
Given a morphism $f \colon T \to S$, a morphism $g \colon T \to \mathbf P(\mathscr F)$ is classified by the surjection $f^*\mathscr F \twoheadrightarrow g^*\mathscr O_{\mathbf P(\mathscr F)}(1)$. By \autoref{Lem ses}, the subscheme $g^*\mathbf P(\mathscr G)$ is cut out by the image of the composite
\begin{equation}
f^*\mathscr E \otimes g^*\mathscr O_{\mathbf P(\mathscr F)}(-1) \to f^*\mathscr F \otimes g^*\mathscr O_{\mathbf P(\mathscr F)}(-1) \twoheadrightarrow \mathcal O_T,\label{Eq composite}
\end{equation}
so~$g$ lands in $\mathbf P(\mathscr F) \setminus \mathbf P(\mathscr G)$ if and only if the map \eqref{Eq composite} is surjective. By the final statement of \autoref{Lem ses} and since $\mathcal O_{\mathbf P(\mathscr G)}$ is flat over~$S$, the map \eqref{Eq composite} is always injective, so $\mathbf P(\mathscr F) \setminus \mathbf P(\mathscr G)$ classifies the maps $f^*\mathscr F \twoheadrightarrow \mathscr L$ onto a line bundle such that $f^*\mathscr E \to f^*\mathscr F \to \mathscr L$ is an isomorphism. These are exactly the sections of $f^*\mathscr E \to f^*\mathscr F$, proving the first statement. It is clear that this is a torsor under the functor taking $f \colon T \to S$ to $\Hom_T(f^*\mathscr G,f^*\mathscr E)$, which in our conventions is represented by $\mathbf A(\mathscr E^{-1} \otimes \mathscr G)$.
\end{proof}

\begin{Lemma}\label{Lem divisor P(E)}
Let~$S$ be an integral scheme, let~$\mathscr E$ be a locally free sheaf on~$S$ of rank~$n$, and let $X = \mathbf P(\mathscr E)$ with its projection $\pi \colon X \to S$. Then the effective Cartier divisors~$D$ on~$X$ with $\mathcal O_X(D) \simeq \pi^*\mathscr L^{-1} \otimes \mathcal O_{\mathbf P(\mathscr E)}(1)$ for a line bundle $\mathscr L \in \Pic(S)$ are given by the inclusions $\mathbf P(\mathscr F) \hookrightarrow \mathbf P(\mathscr E)$ induced by short exact sequences $0 \to \mathscr L \to \mathscr E \to \mathscr F \to 0$. If $n \geq 2$, then~$D$ is irreducible if and only if~$\mathscr F$ is torsion-free.
\end{Lemma}

\begin{proof}
Effective Cartier divisors~$D$ with $\mathcal O_X(D) \simeq \pi^*\mathscr L^{-1} \otimes \mathcal O_{\mathbf P(\mathscr E)}(1)$ correspond to nonzero maps $\pi^*\mathscr L \to \mathcal O_{\mathbf P(\mathscr E)}(1)$. By the adjunction $\pi^* \dashv \pi_*$, these correspond to nonzero maps $\mathscr L \to \pi_*\mathcal O_{\mathbf P(\mathscr E)}(1) \simeq \mathscr E$, where the isomorphism is \cite[Prop.~2.1.15]{EGA3I}. Any such map is injective because~$S$ is integral, hence corresponds to a short exact sequence
\[
0 \to \mathscr L \to \mathscr E \to \mathscr F \to 0.
\]
The first statement therefore follows from \autoref{Lem ses}. For the second, recall that \makebox{$\Pic(X) \cong \pi^*\Pic(S) \oplus \mathbf Zh$} when $n \geq 2$. Thus, any expression $D = D_1 + D_2$ as a sum of effective divisors has $\mathcal O_X(D_i) \simeq \pi^*\mathscr L_i^{-1} \otimes \mathcal O_{\mathbf P(\mathscr E)}(d_i)$ for some $d_i \in \mathbf Z$ and $\mathscr L_i \in \Pic(S)$. Since the~$D_i$ are effective, we get $d_i \geq 0$, so one is~$0$ and the other is~$1$. We see that~$D$ is irreducible if and only if the map $\mathscr L \hookrightarrow \mathscr E$ does not factor via a strict inclusion $\mathscr L \hookrightarrow \mathscr M$ for some line bundle~$\mathscr M$, which means exactly that~$\mathscr F$ is torsion-free.
\end{proof}

The following result is probably known, but I am not aware of a reference.

\begin{Lemma}\label{Lem projection blowup}
Let~$S$ be a scheme, let $0 \to \mathscr E \to \mathscr F \to \mathscr G \to 0$ be a short exact sequence of vector bundles on~$S$, and write $\pi_\mathscr E \colon \mathbf P(\mathscr E) \to S$ for the projection, and likewise for~$\pi_\mathscr F$ and~$\pi_\mathscr G$. Let~$\mathscr H$ be the vector bundle on~$\mathbf P(\mathscr E)$ defined as the pushout $\pi_\mathscr E^*\mathscr F \oplus_{\pi_\mathscr E^*\mathscr E} \mathcal O_{\mathbf P(\mathscr E)}(1)$. Then the blowup $\widetilde X = \Bl_{\mathbf P(\mathscr G)}(\mathbf P(\mathscr F))$ of $X = \mathbf P(\mathscr F)$ along the closed subscheme~$\mathbf P(\mathscr G)$ is canonically isomorphic to the projective bundle $Y = \mathbf P_{\mathbf P(\mathscr E)}(\mathscr H)$ over~$\mathbf P(\mathscr E)$, taking the exceptional divisor in~$\widetilde X$ to the divisor $\mathbf P_{\mathbf P(\mathscr E)}(\pi_\mathscr E^*\mathscr G) \subseteq Y$.
\end{Lemma}

When~$S$ is the spectrum of a field, this is \cite[Prop.~9.11]{EisenbudHarris} (note that \cite{EisenbudHarris} uses the dual convention for the projective bundle of a vector bundle), whose proof can be adapted to prove the result above for general~$S$ under the additional hypothesis that the short exact sequence splits. One could deduce the general result from this using local splitting and glueing. Instead, we only use universal properties of projective bundles and blowups. A similar strategy was used in \cite[Appendix]{vDdBPaulsen} to get a coordinate-free version of the theory of elementary transformations of vector bundles of Maruyama.

\begin{proof}
We will use repeatedly and without mention that a short exact sequence of vector bundles stays exact after pullback along an arbitrary morphism of schemes. Note that~$\mathscr H$ sits in a commutative diagram with exact rows
\begin{equation}
\begin{tikzcd}[row sep=1.2em,column sep=1.2em]
0 \ar{r} & \pi_\mathscr E^*\mathscr E \ar{r}\ar[two heads]{d} & \pi_\mathscr E^*\mathscr F \ar{r}\ar[two heads]{d} & \pi_\mathscr E^*\mathscr G \ar{r}\ar[equal]{d} & 0 \\
0 \ar{r} & \mathcal O_{\mathbf P(\mathscr E)}(1) \ar{r} & \mathscr H \ar{r} & \pi_\mathscr E^*\mathscr G \ar{r} & 0
\end{tikzcd}\label{Dia pushout}
\end{equation}
where the left square is a pushout, which in particular shows that it is a vector bundle. Replacing~$S$ by a clopen subscheme, we may assume that the ranks of~$\mathscr E$, $\mathscr F$, and~$\mathscr G$ are constant. If~$\mathscr E$ is zero, then~$\mathbf P(\mathscr E)$ is empty and $\mathbf P(\mathscr F) = \mathbf P(\mathscr G)$, so the blowup $\Bl_{\mathbf P(\mathscr G)}(\mathbf P(\mathscr F))$ is empty as well. If $\mathscr G = 0$, then~$\mathbf P(\mathscr G)$ is empty and~$\mathscr H$ is a line bundle, so both sides are $\mathbf P(\mathscr E) = \mathbf P(\mathscr G)$. Thus, we may assume that~$\mathscr E$ and~$\mathscr G$ are both nonzero.

We have a tautological quotient $\pi_\mathscr F^*\mathscr F \twoheadrightarrow \mathcal O_{\mathbf P(\mathscr F)}(1)$, and the image of the composition $\pi_\mathscr F^*\mathscr E \otimes \mathcal O_{\mathbf P(\mathscr F)}(-1) \to \pi_\mathscr F^*\mathscr F \otimes \mathcal O_{\mathbf P(\mathscr F)}(-1) \twoheadrightarrow \mathcal O_{\mathbf P(\mathscr F)}$ is the ideal sheaf cutting out~$\mathbf P(\mathscr G)$ by \autoref{Lem ses}. Writing $\pi \colon \widetilde X \to X$ for the blowup and $f \colon \widetilde X \to S$ for the structure map, we see that the image of the map $f^*\mathscr E \otimes \pi^*\mathcal O_{\mathbf P(\mathscr F)}(-1) \to \mathcal O_{\widetilde X}$ is an invertible sheaf of ideals, which we will denote by~$\mathcal I$. This gives a surjection $f^*\mathscr E \twoheadrightarrow \mathcal I \otimes \pi^*\mathcal O_{\mathbf P(\mathscr F)}(1)$, defining a map $\phi \colon \widetilde X \to \mathbf P(\mathscr E)$ over~$S$ with $\phi^*\mathcal O_{\mathbf P(\mathscr E)}(1) \cong \mathcal I \otimes \pi^*\mathcal O_{\mathbf P(\mathscr F)}(1)$. Now the map $f^*\mathscr F \twoheadrightarrow \mathcal \pi^*\mathcal O_{\mathbf P(\mathscr F)}(1)$ factors as $f^*\mathscr F \twoheadrightarrow \phi^*\mathscr H \twoheadrightarrow \pi^*\mathcal O_{\mathbf P(\mathscr F)}(1)$ since its restriction to~$f^*\mathscr E$ lands in $\phi^*\mathcal O_{\mathbf P(\mathscr E)}(1) \subseteq \pi^*\mathcal O_{\mathbf P(\mathscr F)}(1)$. The surjection $\phi^*\mathscr H \twoheadrightarrow \pi^*\mathcal O_{\mathbf P(\mathscr F)}(1)$ defines a morphism $\widetilde\phi \colon \widetilde X \to Y$ over $\mathbf P(\mathscr E)$ such that $\widetilde\phi^*\mathcal O_{\mathbf P(\mathscr H)}(1) \cong \pi^*\mathcal O_{\mathbf P(\mathscr F)}(1)$.

Conversely, write $\pi_\mathscr H \colon Y \to \mathbf P(\mathscr E)$ and $g \colon Y \to S$ for the structure maps. We get a tautological quotient $\pi_\mathscr H^*\mathscr H \twoheadrightarrow \mathcal O_{\mathbf P(\mathscr H)}(1)$, whose composition with the surjective map $g^*\mathscr F \twoheadrightarrow \pi_\mathscr H^*\mathscr H$ defines a morphism $\psi \colon Y \to X$ with $\psi^*\mathcal O_{\mathbf P(\mathscr F)}(1) \cong \mathcal O_{\mathbf P(\mathscr H)}(1)$. Moreover, the map $\pi_{\mathscr H}^*\mathcal O_{\mathbf P(\mathscr E)}(1) \otimes \mathcal O_{\mathbf P(\mathscr H)}(-1) \to \mathcal O_{\mathbf P(\mathscr H)}$ is injective by \autoref{Lem ses} applied to the bottom sequence of \eqref{Dia pushout}. Thus, the image of the composite map
\[
g^*\mathscr E \otimes \mathcal O_{\mathbf P(\mathscr H)}(-1) \hookrightarrow g^*\mathscr F \otimes \mathcal O_{\mathbf P(\mathscr H)}(-1) \twoheadrightarrow \pi_\mathscr H^*\mathscr H \otimes \mathcal O_{\mathbf P(\mathscr H)}(-1) \twoheadrightarrow \mathcal O_{\mathbf P(\mathscr H)}
\]
is the invertible sheaf $\pi_\mathscr H^*\mathcal O_{\mathbf P(\mathscr E)}(1) \otimes \mathcal O_{\mathbf P(\mathscr H)}(-1)$. The universal property of blowing up shows that $\psi \colon Y \to X$ lifts uniquely to a map $\widetilde\psi \colon Y \to \widetilde X$ such that the image of \makebox{$\widetilde\psi^*\mathcal I \to \mathcal O_Y$} is the ideal $\pi_\mathscr H^*\mathcal O_{\mathbf P(\mathscr E)}(1) \otimes \mathcal O_{\mathbf P(\mathscr H)}(-1)$ (which cuts out the divisor $\mathbf P(\pi_\mathscr E^*\mathscr G) \subseteq Y$ by \autoref{Lem ses}). Using the universal properties, it is straightforward to see that~$\widetilde\phi$ and~$\widetilde\psi$ are inverses.
\end{proof}

\begin{Cor}\label{Cor projection torsor}
Let~$S$ be a scheme, let $0 \to \mathscr E \to \mathscr F \to \mathscr G \to 0$ be a short exact sequence of vector bundles on~$S$, and write $\pi_\mathscr E \colon \mathbf P(\mathscr E) \to S$ for the projection, and likewise for~$\pi_\mathscr F$ and~$\pi_\mathscr G$. Then the linear projection $\mathbf P(\mathscr F) \setminus \mathbf P(\mathscr G) \to \mathbf P(\mathscr E)$ is canonically an $\mathbf A(\pi_\mathscr E^*\mathscr G \otimes \mathcal O_{\mathbf P(\mathscr E)}(-1))$-torsor.
\end{Cor}

\begin{proof}
This is immediate from \autoref{Lem projection blowup} and \autoref{Lem torsor}.
\end{proof}

\section{The criterion}\label{Sec criterion}
Let~$\mathscr E$ be a vector bundle of rank~$n \geq 2$ on~$\mathbf P^1$. Then $X = \mathbf P(\mathscr E)$ is a smooth projective variety of dimension~$n$ with a projection $\pi \colon X \to \mathbf P^1$, and $\Pic(X) \cong \mathbf Zf \oplus \mathbf Zh$ where $f = \pi^*(\mathcal O_{\mathbf P^1}(1))$ is a fibre and $h = \mathcal O_{\mathbf P(\mathscr E)}(1)$.

Let $Z \subseteq X$ be a closed subscheme with complement~$U$. We will study necessary and sufficient conditions for~$U$ to be isomorphic to~$\mathbf A^n$. Firstly,~$Z$ is a divisor by \cite[Tag~\href{https://stacks.math.columbia.edu/tag/0BCW}{0BCW}]{Stacks}. If~$Z_1,\ldots,Z_s$ are the irreducible components of~$Z$, the exact sequence
\begin{align*}
\Gamma(X,\mathcal O_X^\times) \to \Gamma(U,\mathcal O_U^\times) \to \mathbf Z^s &\to \Pic(X) \to \Pic(U) \to 0 \\
e_i &\mapsto Z_i
\end{align*}
shows that $Z_1, \ldots, Z_s$ form a basis of~$\Pic(X)$, so in particular $s = 2$. Thus, it remains to study $X \setminus (D \cup E)$ for reduced and irreducible divisors~$D$ and~$E$.

When~$D$ and~$E$ are of the form~$df + h$ and~$ef + h$ for $d, e \in \mathbf Z$, we give a complete characterisation for when the complement $U = X \setminus (D \cup E)$ is isomorphic to~$\mathbf A^n$. By \autoref{Lem divisor P(E)}, there are short exact sequences
\begin{alignat*}{5}
0 & \to &\ \mathscr L & \to &\ \mathscr E & \to &\ \mathscr F & \to &\ 0,\\
0 & \to &\ \mathscr M & \to &\ \mathscr E & \to &\ \mathscr G & \to &\ 0\makebox*{,}{}
\end{alignat*}
of vector bundles on~$\mathbf P^1$ such that $D = \mathbf P(\mathscr F)$ and $E = \mathbf P(\mathscr G)$. Then $D \cap E$ is given by~$\mathbf P(\mathscr H)$, where~$\mathscr H$ is the cokernel of $\mathscr L \oplus \mathscr M \to \mathscr E$. Because the short exact sequence $0 \to \mathscr H\tors \to \mathscr H \to \mathscr H\tf \to 0$ splits, this means that $D \cap E$ is the union of the projective bundle $\mathbf P(\mathscr H\tf)$ and the scheme $\mathbf P(\mathscr H\tors)$. When $D \neq E$, the fibres~$D_p$ and~$E_p$ differ for general $p \in \mathbf P^1$. Then~$\mathscr H\tf$ has rank~$n-2$ (in particular, $\mathbf P(\mathscr H\tf) = \varnothing$ when $n = 2$), and~$\mathbf P(\mathscr H\tors)$ is the restriction $D|_W = E|_W \to W$ over the scheme-theoretic locus $W \subseteq \mathbf P^1$ of~$R$-points $p \colon \Spec R \to \mathbf P^1$ where $p^*D = p^*E$. Let~$r$ be the number of points in $\operatorname{Supp}(\mathscr H\tors)$; in other words, the number of points $p \in \mathbf P^1$ such that the fibres~$D_p$ and~$E_p$ agree.

\begin{Thm}\label{Thm criterion}
Let $X = \mathbf P(\mathscr E) \to \mathbf P^1$ be as above, let~$D$ and~$E$ be irreducible divisors with $D \sim df + h$ and $E \sim ef + h$, and let $U = X \setminus (D \cup E)$. Then $U \cong \mathbf A^n$ if and only if $e = d \pm 1$ and~$r=1$.
\end{Thm}

\begin{proof}
Firstly, assume $U \cong \mathbf A^n$. Then~$D$ and~$E$ form a basis of $\Pic(X) \cong \mathbf Zf \oplus \mathbf Zh$, so $e = d \pm 1$. Note that~$\mathbf P(\mathscr H)$ is the union of~$\mathbf P(\mathscr H\tf)$ and $\mathbf P(\mathscr H\tors) = \mathbf P(\mathscr F)|_W$, which intersect in~$\mathbf P(\mathscr H\tf)|_W$. Thus, in the Grothendieck ring of varieties, we have
\begin{align*}
[\mathbf P(\mathscr H)] &= [\mathscr P(\mathscr H\tf)] + [\mathscr P(\mathscr F)|_W] - [\mathscr P(\mathscr H\tf)|_W] \\
&= (1 + \mathbf L)(1 + \ldots + \mathbf L^{n-3}) + r(1 + \ldots + \mathbf L^{n-2}) - r(1 + \ldots + \mathbf L^{n-3}) \\
&= (1 + \mathbf L)(1 + \ldots + \mathbf L^{n-3}) + r\mathbf L^{n-2}.
\end{align*}
Thus, we compute
\begin{align*}
[U] =&\ [X]-[D]-[E]+[D \cap E] \\
=&\ (1\!+\!\mathbf L)(1\!+\!\ldots\!+\!\mathbf L^{n-1}) - 2(1\!+\!\mathbf L)(1\!+\!\ldots\!+\!\mathbf L^{n-2}) + (1\!+\!\mathbf L)(1\!+\!\ldots\!+\!\mathbf L^{n-3}) + r\mathbf L^{n-2} \\
=&\ (1+\mathbf L)(\mathbf L^{n-1}-\mathbf L^{n-2}) + r\mathbf L^{n-2} = \mathbf L^n + (r-1)\mathbf L^{n-2}.
\end{align*}
Since $[U] = \mathbf L^n$, we conclude that $r = 1$.

Conversely, suppose that $e = d \pm 1$ and $r = 1$. Without loss of generality, we may assume $e = d+1$. Let~$\mathscr E'$ be the kernel of the composite surjection $\mathscr E \twoheadrightarrow \mathscr H \twoheadrightarrow \mathscr H\tf$, and let $X' = \mathbf P(\mathscr E')$. By \autoref{Cor projection torsor}, linear projection away from $\mathbf P(\mathscr H\tf) \subseteq \mathbf P(\mathscr E)$ defines an $\mathbf A(\mathscr K)$-torsor $\phi \colon X \setminus \mathbf P(\mathscr H\tf) \to X'$, where $\mathscr K = \pi_{\mathscr E'}^*\mathscr H\tf \otimes \mathcal O_{\mathbf P(\mathscr E')}(-1)$.

The surjection $\mathscr H \twoheadrightarrow \mathscr H\tf$ gives an injection $\mathscr L \oplus \mathscr M \hookrightarrow \mathscr E'$. By \autoref{Lem divisor P(E)}, the inclusions $\mathscr L \hookrightarrow \mathscr E'$ and $\mathscr M \hookrightarrow \mathscr E'$ define divisors $D', E' \subseteq \mathbf P(\mathscr E')$ whose pullbacks along~$\phi$ are the restrictions of~$D$ and~$E$ to $X \setminus \mathbf P(\mathscr H\tf)$. Note that both maps $X \hookleftarrow X \setminus \mathbf P(\mathscr H\tf) \twoheadrightarrow X'$ induce isomorphisms on~$\Pic$, since~$\mathbf P(\mathscr H\tf)$ has codimension~$2$ and~$\phi$ is an $\mathbf A^{n-2}$-bundle. Since $\phi^*\mathcal O_{\mathbf P(\mathscr E')}(1) \cong \mathcal O_{\mathbf P(\mathscr E)}(1)|_{X \setminus \mathbf P(\mathscr H\tf)}$, we get $D' \sim df'+h'$ and $E' \sim ef'+h'$ under the analogous isomorphism $\Pic(X') \cong \mathbf Zf' \oplus \mathbf Zh'$.

Suppose we know that $U' \cong \mathbf A^2$. Since~$\mathscr H\tf$ is a sum of line bundles and $\Pic(\mathbf A^2) = 0$, we see that~$\mathscr K|_{U'}$ is trivial. Moreover, the $\mathscr K|_{U'}$-torsor $U \to U'$ is trivial since $U' = \mathbf A^2$ is affine. Thus, we conclude that $U \cong \mathbf A^{n-2} \times U' \cong \mathbf A^n$ (this can also be deduced from \cite[Thm.~4.4]{BassConnellWright}, although the more precise computation above is probably an easier argument overall). Replacing~$\mathscr E$ by~$\mathscr E'$, we have therefore reduced to the case $n = 2$, and we will drop all the primes from the notation.

By assumption, there is a unique $p \in \mathbf P^1$ such that the points~$D_p$ and~$E_p$ in $\pi^{-1}(p) \cong \mathbf P^1$ agree. Let $C \subseteq X$ be the fibre above~$p$, and let~$q \in C$ be the intersection point $D \cap E$. Note that $C+D \sim (d+1)f + h \sim E$, so the pencil $a(C+D) + bE$ defines a rational map $\psi \colon X \dashrightarrow \mathbf P^1$ such that $\psi^{-1}([0:1]) = C+D$ and $\psi^{-1}([1:0]) = E$. Set $m = E^2$, and note that $m \geq 2$ since $m-1 = D \cdot E \geq 1$. If $X_1 \to X$ is the blowup at~$q$, then the strict transforms~$\widetilde D$ and~$\widetilde E$ intersect with multiplicity~$m-2$, and all intersection points lie on the exceptional divisor~$E_1$. Since~$D$ and~$E$ are smooth, their strict transforms meet~$E_1$ once, so~$\widetilde D$ and~$\widetilde E$ intersect once with multiplicity~$m-2$. The rational map $\psi_1 \colon X_1 \dashrightarrow \mathbf P^1$ corresponds to the pencil $a(\widetilde C + \widetilde D + E_1)+b\widetilde E$, so~$E_1$ maps to~$[0:1]$.  Repeating this process $i < m-1$ times gives strict transforms~$\widetilde D$ and~$\widetilde E$ intersecting once with multiplicity~$m-1-i$, and after~$m-1$ steps they no longer meet. For $i < m$, the fibre of~$\psi_i$ above~$[0:1]$ is $\widetilde C + \widetilde D + \widetilde E_1 + \ldots + \widetilde E_{i-1} + E_i$. After~$m-1$ steps, the strict transform~$\widetilde E$ meets the last exceptional~$E_{m-1}$ with multiplicity~$1$ but does not meet the other components of $\psi_{m-1}^{-1}([0:1])$, and blowing up this point resolves the rational map~$\psi$ to a morphism $\psi_m \colon X_m \to \mathbf P^1$ with $\psi_m^{-1}([0:1]) = \widetilde C + \widetilde D + \widetilde E_1 + \ldots + \widetilde E_{m-1}$. The dual graph of the preimage of $D \cup E$ is given by
\[
\begin{tikzcd}[row sep=.3em,column sep=.25em]
 & & & \widetilde D \ar[-,start anchor=center,end anchor=center,shorten <=10pt,shorten >=18pt]{rd} \\
 & & & & \widetilde E_{m-1} \ar[-]{rrr} & & & \cdots \ar[-]{rrr} & & & \widetilde E_1\punct{,} \\
\widetilde E \ar[-]{rrr} & & & E_m \ar[-,start anchor=center,end anchor=center,shorten <=10pt,shorten >=18pt]{ru}
\end{tikzcd}
\]
with self-intersections
\[
\begin{tikzcd}[row sep=.3em,column sep=.25em]
 & & & -1 \ar[-,start anchor=center,end anchor=center,shorten <=10pt,shorten >=10pt]{rd} \\
 & & & & -2 \ar[-]{rrr} & & & \cdots \ar[-]{rrr} & & & -2\punct{.} \\
0 \ar[-]{rrr} & & & -1 \ar[-,start anchor=center,end anchor=center,shorten <=10pt,shorten >=10pt]{ru}
\end{tikzcd}
\]
We may thus successively blow down~$\widetilde D$, $\widetilde E_{m-1}$, $\ldots$, $\widetilde E_1$, leaving us with a smooth surface~$X'$. Since these curves are also contracted by~$\psi_m$, we get a map $\psi' \colon X' \to \mathbf P^1$, all of whose fibres are smooth rational curves. We conclude that~$\psi'$ is a $\mathbf P^1$-bundle over~$\mathbf P^1$. Since the image of~$\widetilde E$ (resp.~$E_m$) is a fibre (resp.\ section) of~$\psi'$, we conclude that their complement is an $\mathbf A^1$-bundle over~$\mathbf A^1$. This proves that $U \cong \mathbf A^2$.
\end{proof}

\appendix
\titleformat{\section}[block]{\Large\bfseries\scshape\filcenter}{Appendix \thesection.}{1ex}{}
\section{A coordinate-free counterexample to the Jacobian conjecture}In this appendix, we give a coordinate-free exposition of the first counterexample to the Jacobian conjecture in characteristic~$0$ \cite{AlpogeX}, \cite{JiangX}. This appendix is a version of the blog post \cite{vDdBBlog2} and the expository paper \cite{vDdBExpo} by the author. We include it here because it does not yet seem to appear in the published literature.

\begin{Def}\label{Def}
Let $\varphi \colon \mathbf P^1 \times \Sym^2(\mathbf P^1) \to \Sym^3(\mathbf P^1)$ be the map $(p,\{q,r\}) \mapsto \{p,q,r\}$. Let $R \subseteq \mathbf P^1 \times \Sym^2(\mathbf P^1)$ be the ramification divisor, and let $H \subseteq \Sym^3(\mathbf P^1)$ be a hyperplane. Let $U = (\mathbf P^1 \times \Sym^2(\mathbf P^1)) \setminus (R \cup \varphi^{-1}(H))$, and let $V = \Sym^3(\mathbf P^1) \setminus H$.
\end{Def}

\begin{Thm}[\cite{AlpogeX,JiangX}]\label{Thm Jacobian}
If~$H$ is tangent to order~$2$ to the small diagonal, then the restriction $\varphi|_U \colon U \to V$ is an \'etale map between affine spaces of dimension~$3$ that is not an isomorphism.
\end{Thm}

The theory of symmetric polynomials identifies $\Sym^n(\mathbf P^1)$ with the projective space~$\mathbf P^n$, so~$U$ and~$V$ are open subvarieties of $\mathbf P^1 \times \mathbf P^2$ and~$\mathbf P^3$ respectively. As was noted in various places (see for instance \cite{SBS,LevinsonBlog,vDdBBlog1}), it is clear that $V \cong \mathbf A^3$, that~$\varphi|_U$ is \'etale (because we removed the ramification locus~$R$), and that~$\varphi|_U$ is not an isomorphism (since it has generic degree~$3$).

Thus, the content of the theorem above lies in proving that $U \cong \mathbf A^3$ if~$H$ is tangent to order~$2$ (i.e., tangent but not osculating) to the small diagonal. This can be done by explicit computation, but this does not shed light on the underlying geometry. Instead, we will deduce this from the much more general criterion of \autoref{Thm criterion}. In particular, this also provides a converse; see \autoref{Thm affine}.

We start by studying the geometry of the quotient map $(\mathbf P^1)^n \to \Sym^n(\mathbf P^1)$, which we then use to describe the properties of the divisors~$R$ and~$\varphi^{-1}(H)$.

\subsection{Some lemmas on symmetric powers}
The isomorphism $U \cong \mathbf A^3$ in \autoref{Thm Jacobian} relies on the tangency of the hyperplane $H \subseteq \Sym^3(\mathbf P^1)$ to the small diagonal. We need a local computation to extract the geometric implication of this hypothesis; see \autoref{Cor tangent} below.

By the theory of symmetric polynomials, the quotient $(\mathbf P^1)^n \to \Sym^n(\mathbf P^1)$ is given by
\[
\pi = [\bar\sigma_n : \ldots : \bar\sigma_0 ] \colon (\mathbf P^1)^n \to \mathbf P^n,
\]
where~$\bar\sigma_i$ is the multi-homogenised elementary symmetric polynomial
\[
\bar\sigma_i([x_1:y_1],\ldots,[x_n:y_n]) = \sum_{\substack{I \subseteq \{1,\ldots,n\}\\ \lvert I \rvert = i}} \prod_{j \in I} x_j \cdot \prod_{j \not\in I} y_j.
\]
Since~$\pi$ is homogeneous of degree $(1,\ldots,1)$, we get $\pi^*\mathcal O_{\mathbf P^n}(1) \simeq \mathcal O_{(\mathbf P^1)^n}(1,\ldots,1)$. We also note that the image of the small diagonal in~$(\mathbf P^1)^n$ gives (a slightly reparametrised version of) the rational normal curve $\iota \colon \mathbf P^1 \to \mathbf P^n$, where the~$i$-th coordinate is $\binom ni x^{n-i}y^i$. We denote the image of~$\iota$ by~$C$.

\begin{Lemma}\label{Lem p_i}
Let $\pi \colon (\mathbf P^1)^n \to \mathbf P^n$ be the map above. Then the preimage of the linear subspace $\{[x_0:\ldots:x_n]\ |\ x_0 = \ldots = x_i = 0\}$ is the locus in~$(\mathbf P^1)^n$ where at least~$i+1$ of the coordinates are equal to~$[0:1]$.
\end{Lemma}

\begin{proof}
The coordinates $([x_1:y_1],\ldots,[x_n:y_n])$ in~$(\mathbf P^1)^n$ satisfy the relation
\[
\sum_{i=0}^n \bar\sigma_{n-i}T^i = (x_1+y_1T)\cdots(x_n+y_nT).
\]
The left hand side shows that $\bar\sigma_n = \ldots = \bar\sigma_{n-i} = 0$ if and only if this polynomial is divisible by~$T^{i+1}$, which by examining the right hand side means exactly that at least~$i+1$ of the coordinates are equal to~$[0:1]$.
\end{proof}

\begin{Cor}\label{Cor tangent}
A hyperplane $H \subseteq \mathbf P^n$ is tangent to~$C$ to order~$\geq i$ at~$p$ if and only if~$\pi^{-1}(H)$ contains the locus where at least~$n-i+1$ of the coordinates are equal to~$p$.
\end{Cor}

\begin{proof}
Applying a coordinate transformation, we may assume $p = [0:1]$. If~$H$ is given by $a_0x_0 + \ldots + a_nx_n = 0$, then its pullback along the rational normal curve $\mathbf P^1 \to \mathbf P^n$ is given by $\sum_{i=0}^n a_i \binom n i x^{n-i}y^i$, so it is tangent to order~$\geq i$ at~$[0:1]$ if and only if $a_n = \ldots = a_{n-i+1} = 0$. This means that~$H$ contains the linear subspace $V(x_0,\ldots,x_{n-i})$, which by \autoref{Lem p_i} means that~$\pi^{-1}(H)$ contains the locus where at least~$n-i+1$ of the coordinates are equal to~$p$.
\end{proof}

\subsection{Geometric properties of the boundary divisors}
To study $(\mathbf P^1 \times \mathbf P^2) \setminus (R \cup \varphi^{-1}(H))$, we first give a geometric description of the boundary divisors~$R$ and~$\varphi^{-1}(H)$, as well as their intersection $R \cap \varphi^{-1}(H)$.

\begin{Lemma}\label{Lem R}
The divisor~$R$ is irreducible of degree~$(2,1)$, and its pullback to~$(\mathbf P^1)^3$ is the union of pairwise diagonals $\Delta_{12} \cup \Delta_{13}$.
\end{Lemma}

\begin{proof}
The cover $(\mathbf P^1)^3 \to \mathbf P^3$ is ramified along the union of the pairwise diagonals~$\Delta_{ij}$, which each have ramification index~$2$ since their inertia subgroups are~$\langle(ij)\rangle$. Using multiplicativity of ramification indices \cite[Tag \href{https://stacks.math.columbia.edu/tag/0BRL}{0BRL}]{Stacks}, we see that the preimage of~$R$ in~$(\mathbf P^1)^3$ is $\Delta_{12} \cup \Delta_{13}$, which has degree $(1,1,0) + (1,0,1) = (2,1,1)$. Thus,~$R$ has degree~$(2,1)$, since the pullback of $\mathcal O_{\mathbf P^2}(1)$ to~$(\mathbf P^1)^2$ is $\mathcal O_{(\mathbf P^1)^2}(1,1)$. Since the Galois group of $(\mathbf P^1)^3 \to \mathbf P^1 \times \mathbf P^2$ interchanges the two components of $\Delta_{12} \cup \Delta_{13}$, we see that~$R$ is irreducible.
\end{proof}

Given a subscheme $Z \subseteq \mathbf P^1 \times \mathbf P^2$ and a point $p \in \mathbf P^1$, we will write $Z_p \subseteq \mathbf P^2$ for the fibre $Z \cap (\{p\} \times \mathbf P^2)$ above~$p$.

\begin{Lemma}\label{Lem H}
The divisor~$\varphi^{-1}(H)$ has degree~$(1,1)$. Given a point $p \in \mathbf P^1$, the fibre~$\varphi^{-1}(H)_p$ equals~$R_p$ (resp.~$\mathbf P^2$) if and only if~$H$ is tangent to order~$2$ (resp.~$3$) at~$p$.
\end{Lemma}

\begin{proof}
The degree statement follows since the pullback of $\mathcal O_{\mathbf P^1 \times \mathbf P^2}(a,b)$ along the quotient map $(\mathbf P^1)^3 \to \mathbf P^1 \times \mathbf P^2$ is $\mathcal O_{(\mathbf P^1)^3}(a,b,b)$, and~$\pi^{-1}(H)$ has degree $(1,1,1)$. To see whether the fibre of~$\varphi^{-1}(H)$ agrees with that of~$R$ or $\mathbf P^1 \times \mathbf P^2$, it suffices to check this after pulling back along the map $g \colon (\mathbf P^1)^3 \to \mathbf P^1 \times \mathbf P^2$. Now \autoref{Lem R} gives $g^*R = \Delta_{12} \cup \Delta_{13}$, whose fibre above~$p$ is $(\{p\} \times \mathbf P^1) \cup (\mathbf P^1 \times \{p\})$. On the other hand, by \autoref{Cor tangent}, the pullback of~$H$ to $\{p\} \times (\mathbf P^1)^2$ contains $(\{p\} \times \mathbf P^1) \cup (\mathbf P^1 \times \{p\})$ (resp. equals~$(\mathbf P^1)^2$) if and only if~$H$ is tangent to order~$\geq 2$ (resp.~$\geq 3$) at~$p$. In the order~$2$ case, the inclusion $R_p \hookrightarrow \varphi^{-1}(H)_p$ is an equality since both are effective divisors of degree~$1$ on~$\mathbf P^2$.
\end{proof}

\begin{Thm}\label{Thm affine}
The complement $U = \mathbf P^1 \times \mathbf P^2 \setminus (R \cup \varphi^{-1}(H))$ is isomorphic to~$\mathbf A^3$ if and only if~$H$ is tangent but not osculating to the small diagonal.
\end{Thm}

\begin{proof}
By \autoref{Lem R} and \autoref{Lem H}, the divisors~$R$ and~$\varphi^{-1}(H)$ have degrees~$(2,1)$ and~$(1,1)$ respectively. If~$H$ is osculating to the small diagonal (i.e., tangent to order~$3$) at a point~$p$, then~$\varphi^{-1}(H)$ contains $\{p\} \times \mathbf P^2$ by \autoref{Lem H}, hence is reducible. Then the components of~$R$ and~$\varphi^{-1}(H)$ cannot form a basis of $\Pic(\mathbf P^1 \times \mathbf P^2)$, so~$U$ is not isomorphic to~$\mathbf A^3$. In all other cases, $H$ does not contain a vertical fibre $\{p\} \times \mathbf P^2$, hence is irreducible by \autoref{Lem divisor P(E)}. By \autoref{Lem R}, we know that~$R$ is irreducible. By \autoref{Thm criterion}, we see that $U \cong \mathbf A^3$ if and only if $R \cap \varphi^{-1}(H)$ is a union of a section and a line $\{p\} \times L$ in a vertical fibre. By \autoref{Lem H}, this happens if and only if~$H$ is tangent to the small diagonal at a point~$p$ (of which there can be at most one since~$H$ intersects the small diagonal with multiplicity~$3$).
\end{proof}

\bibliographystyle{alphaurledit}
{\phantomsection\footnotesize\bibliography{Affine.bib}}
\addcontentsline{toc}{section}{References}

\end{document}